\documentclass[11pt]{article}

\usepackage{latexsym}
\usepackage{amssymb}
\usepackage{amsmath}
\usepackage{amsthm}
\usepackage[T1]{fontenc}
\usepackage[all]{xy}
\usepackage{color}
\usepackage{graphicx}
\usepackage{caption}
\usepackage{indentfirst}      
\usepackage{dsfont}
\usepackage[normalem]{ulem}
\usepackage[cp1251]{inputenc}

\newtheorem{theorem}{Theorem}[section]
\newtheorem{lemma}[theorem]{Lemma}
\theoremstyle{definition}
\newtheorem*{defi}{Acknowledgment}

\newtheorem{proposition}[theorem]{Proposition}

\newtheorem{corollary}[theorem]{Corollary}
\theoremstyle{remark}
\newtheorem{remark}[theorem]{Remark}
\numberwithin{equation}{section}

\newcommand{\R}{\mathbb{R}}
\newcommand{\E}{\mathbb{E}}
\newcommand{\Sp}{\mathbb{S}}
\newcommand{\Hy}{\mathbb{H}}
\newcommand{\Q}{\mathbb{Q}}
\newcommand{\Qe}{\Q^n_\e\times\R}
\newcommand{\Qne}{\Q^n_\e}
\newcommand{\Qee}{\Q_\e}
\newcommand{\Se}{\Sp^n\times\R}

\newcommand{\A}{\mathcal A}

\newcommand{\Nt}{\widetilde\nabla}
\newcommand{\Ntj}{^{f_j}\widetilde\nabla}

\newcommand{\NB}{\overline\nabla}
\newcommand{\Nb}{\bar{N}}

\newcommand{\X}{\mathfrak X}

\newcommand{\ft}{\tilde f}

\newcommand{\e}{\epsilon}
\newcommand{\co}{\colon}
\newcommand{\dt}{\partial_t}
\newcommand{\pl}{\langle}
\newcommand{\pr}{\rangle}
\newcommand{\spa}{\text{span}}

\newcommand{\G}{\Gamma}
\newcommand{\Ric}{\mathrm{Ric}}

\newcommand{\Addresses}{{
\bigskip 
\footnotesize
Institute of Mathematics and Computer Science \newline
University of S\~ao Paulo, S\~ao Carlos, Brazil \\
\textit{E-mail address:} \texttt{estelagarciabr@gmail.com,
manfio@icmc.usp.br}
}}

\title{Einstein submanifolds with parallel mean curvature into $\Se$}

\date{}

\author{Estela Garcia\footnote{The first author is partially supported
by CAPES grant 10869871/D and CNPq  grant 141475/2019-6} and 
Fernando Manfio\footnote{The second author is supported by Fapesp, 
grant 2016/23746-6.}}

\begin{document}

\maketitle

\begin{abstract}
We prove that Einstein submanifolds in $\Sp^n\times\R$ with flat
normal bundle and parallel mean curvature are warped product
of isometric immersions.
\end{abstract}


\section{Introduction}

One of the basic problems in submanifold theory is to provide 
conditions that imply that an isometric immersion of a product
manifold must be a extrinsic product of isometric immersions.
The first contribution to this problem was addressed by Moore
\cite{Mo} for isometric immersions into Euclidean space and
extended by Molzan \cite{Mol} for any space form as ambient
space. They showed that an isometric immersion of a
Riemannian product into a space form must be an extrinsic
product of isometric immersions  of the factors whenever its
second fundamental form is adapted to the product structure
of the manifold.

In this paper we consider isometric immersions of Einstein
manifolds into the product space $\Qe$, with flat normal bundle
and parallel mean curvature vector field. Recall that a 
Riemannian manifold $M$ is an {\em Einstein manifold} if
the Ricci tensor satisfies
\[
\Ric(X,Y) = \lambda\pl X,Y\pr,
\]
for all $X,Y\in\X(M)$ and some constante $\lambda\in\R$.
Riemannian manifolds $M^n_c$ with constant sectional 
curvature $c$ are the simplest examples of Einstein 
Manifolds, where $\lambda=(n-1)c$.

An isometric immersion is said to be Einstein if the submanifold,
endowed with the induced metric, is an Einstein manifold. For
example, Ryan \cite{Ry} gave a local classification of Einstein
hypersurfaces in any space form. In higher codimension,
Di Scala \cite{Sc} proved that Einstein real Kahler submanifolds
of a Euclidean space are totally geodesic provided that they 
are minimal, and same conclusion, by Nolker \cite{No2}, still 
holds for minimal Einstein submanifolds in the Euclidean space, 
with flat normal bundle.

In \cite{On}, Onti classified Einstein submanifolds into space forms,
with flat normal bundle and parallel mean curvature vector field, 
extending a previous result due to Dajczer-Tojeiro \cite{DaTo2}
for isometric immersions of Riemannian manifolds with constant
sectional curvature. More recently, Leandro-Pina-dos Santos
\cite{LePiSa} classified Einstein hypersurfaces into product space
$\Qe$.

In the present paper, we classify Einstein submanifolds
$f\co M^m\to\Se$, $m\geq3$, with flat normal bundle and parallel 
mean curvature vector field. As a first step, we prove that such 
submanifolds have the property that the tangent component of
the second factor of $\Se$ is an eigenvector of all shape operators
of $f$. They belong to a more general class, named {\em class}
$\A$, introduced by Tojeiro \cite{To} for the case of hypersurfaces
and extended to higher codimension by Mendon\c ca-Tojeiro 
\cite{MeTo}. Some important classes of submanifolds of $\Qe$ 
that are included in class $\A$ are hypersurfaces with constant
sectional curvature \cite{MaTo}, rotational hypersurfaces
\cite{DFV}, constant angle hypersurfaces \cite{To} and
biconservative submanifolds \cite{MTU}.

\begin{theorem} \label{theo:introd1}
Let $f\co M^m\to\Qe$, $m\geq3$, be an isometric immersion of a 
connected Einstein manifold, with Ricci curvature $\lambda$ and
flat normal bundle. Then $f$ belongs to class $\A$.
\end{theorem}

Theorem \ref{theo:introd1} extends Lemma 5 of \cite{LePiSa} for
the case of hypersurfaces. Our main result, which extends the 
aforementioned result of Onti \cite{On} is stated as follows:

\begin{theorem} \label{theo:main}
Let $f\co M^m\to\Se$, $m\geq3$, be a proper isometric immersion
of a connected Einstein manifold with Ricci curvature $\lambda$, 
flat normal bundle, and parallel mean curvature vector field $H$.
Then $f$ is a warped product of isometric immersions.
\end{theorem}

The paper is organized as follows. In the next section we just recall
the notion of extrinsic warped product of isometric immersions. In
Section \ref{sec:BasicFacts} we recall the basic equations of an
isometric immersion into $\Qe$. In Section \ref{sec:Flat} we study
submanifolds with flat normal bundle. In particular, we obtain an 
equivalent condition for a submanifold, with such properties, to 
belong to class $\A$. In Section \ref{sec:Ricci} we calculate the
Ricci tensor of a submanifold in $\Qe$, and prove that every
Einstein submanifold with flat normal bundle belongs to class
$\A$. Finally, in the last section we prove a sequence of 
lemmas and Theorem \ref{theo:main}.

\begin{defi}
We are indebted to Ruy Tojeiro and Felippe Guimar\~aes for valuable
suggestions which improved some results in this paper.
\end{defi}

\section{Preliminaries}

A metric $g$ on a product manifold $M^m=M_0\times M_1\times\ldots
\times M_k$ is called a {\em warped product} if there exist
positive functions $\rho_i\in C^\infty(M_0)$, $1\leq i\leq k$, such 
that
\[
g = \pi_0^\ast g_0 + \sum_{i=1}^k(\rho_i\circ\pi_i)^2\pi_i^\ast g_i,
\]
where $\pi_i\co M^m\to M_i$ denotes the canonical projection. We
call $M^m$ endowed with this metric the {\em warped product} of
$M_0,\ldots,M_k$ with warping functions $\rho_1,\ldots,\rho_k$
and denote it by
\[
M^m = M_0\times_{\rho_1}M_1\times\ldots\times_{\rho_k}M_k.
\]

Let $\Qne$ denote a complete and simply connected space form
of constant sectional curvature $\e$. Fix a point $q\in\Qne$ and
let $\Qee^{n_0},\Q^{n_1}_{\e_1},\ldots,\Q^{n_k}_{\e_k}$, with
$n=n_0+n_1+\ldots+n_k$, be submanifolds of $\Qne$ through
$q$ such that the first one is totally geodesic and all the others
are totally umbilical with mean curvature vectors $z_1,\ldots,z_k$
at $q$ and $\pl z_i,z_j\pr=-\e$, for $i\neq j$. The {\em warped
product representation}
\begin{equation} \label{eq:wpr}
\psi\co \Q_\e^{n_0}\times_{\sigma_1}\Q_{\e_1}^{n_1}\times
\ldots\times_{\sigma_k}\Q^{n_k}_{\e_k}\to\Qne
\end{equation}
of $\Qne$ is the map
\[
\psi(p_0,p_1,\ldots,p_k) = p_0 + \sum_{i=1}^k\sigma_i(p_0)(p_i-q),
\]
where the functions $\sigma_i\co\Q_\e^{n_0}\to\R_{+}$ are
defined by
\[
\sigma_i(p) = \left\{
\begin{array}{lcl}
1+\pl p-q,a_i\pr,& \mbox{if} & \e=0 \\
\pl p,a_i\pr, & \mbox{if} & \e\neq 0 
\end{array},
\right.
\]
and satisfy $\sigma_i(q)=1$, with $a_i=cq-z_i$.

It was shown by Nolker \cite{No} that any isometry of a warped
product with $k+1$ factors onto an open dense subset of $\Qne$
arises as the restriction of a warped product representation as
above. Given a warped product representation as in \eqref{eq:wpr}
and isometric immersions $f_i\co M_i\to\Q_{\e_i}^{n_i}$, with
$0\leq i\leq k$ and $\e_0=\e$, the map
\[
f = \psi\circ(f_0\times f_1\times\ldots\times f_k)
\]
is an isometric immersion of the warped product manifold $M^m$
with warping functions $\rho_i=\sigma_i\circ f_0$, called an
{\em extrinsic warped product} of isometric immersions. If all
factors $f_i$ of the extrinsic product are identity maps, the the
map $f$ is called the {\em multi-rotational submanifold}
determined by $\psi$ with $f_0$ as profile.

\section{Basic facts on submanifolds in $\Qe$} 
\label{sec:BasicFacts}

From now on, let $\Q^n_\e$ denote either $\Sp^n$ or $\Hy^n$, 
according to whether $\e=1$ or $\e=-1$, respectively. Given an 
isometric immersion $f\co M^m\to\Qe$, let $\dt$ denote a unit 
vector field tangent to the second 
factor. Then, a tangent vector field $T$ on $M^m$ 
and a normal vector field $\eta$ along $f$ are defined by
\begin{equation} \label{eq:dt}
\dt = f_\ast T+\eta.
\end{equation}
Using that $\dt$ is a parallel vector field in $\Qe$, we obtain by
differentiating equation \eqref{eq:dt} that
\begin{equation} \label{eq:NablaT}
\nabla_XT = A_\eta X
\end{equation}
and
\begin{equation} \label{eq:AlphaT}
\alpha_f(X,T) = - \nabla^\perp_X\eta,
\end{equation}
for all $X\in\X(M)$. Here and in the sequel $A_\eta$ stands for the shape
operator of $f$ in the direction $\eta$, given by
\[
\pl A_\eta X,Y\pr = \pl\alpha_f(X,Y),\eta\pr,
\]
for all $X,Y\in\X(M)$. 
The Gauss, Codazzi and Ricci equations for $f$ are, respectively,
\begin{eqnarray} \label{eq:Gauss}
\begin{aligned}
R(X,Y)Z &= A_{\alpha_f(Y,Z)}X-A_{\alpha_f(X,Z)}Y + \e\big(X\wedge Y \\
&+ \pl X,T\pr Y\wedge T- \pl Y,T\pr X\wedge T\big)Z,
\end{aligned}
\end{eqnarray}
\begin{eqnarray} \label{eq:Codazzi}
\left(\nabla^\perp_X\alpha_f\right)(Y,Z)-\left(\nabla^\perp_Y\alpha_f\right)(X,Z)
=\epsilon\pl(X\wedge Y)T,Z\pr\eta
\end{eqnarray}
and
\begin{eqnarray} \label{eq:Ricci}
R^\perp(X,Y)\xi=\alpha_f(X,A_\xi Y)-\alpha_f(A_\xi X,Y),
\end{eqnarray}
for all $X,Y,Z\in\X(M)$ an for all $\xi\in\G(TM^\perp)$. Equation 
\eqref{eq:Codazzi} can also be written as
\begin{equation} \label{eq:Codazzi2}
(\nabla_XA)(Y,\xi) - (\nabla_YA)(X,\xi) = \e\pl\eta,\xi\pr(X\wedge Y)T,
\end{equation}
where
\[
(X\wedge Y)T = \pl Y,T\pr X-\pl X,T\pr Y,
\]
whereas equation \eqref{eq:Ricci} is equivalent to
\begin{eqnarray} \label{eq:Ricci2}
\pl R^\perp(X,Y)\xi,\zeta\pr = \pl[A_\xi,A_\zeta]X,Y\pr.
\end{eqnarray}

Although this will not be used in this work, it is worth mentioning that
equations \eqref{eq:NablaT}--\eqref{eq:Ricci} completely determine an
isometric immersion of a Riemannian manifold $M^m$ into $\Qe$ up
to isometries of $\Qe$ (see \cite[Corollary 3]{LTV}).

Given an isometric immersion $f\co M^m\to\Qe$, consider $M^m$
as a submanifold in the underlying flat space $\E^{n+2}$, that is, 
consider the isometric immersion $\ft=i\circ f$, where 
$i:\Qe\to\E^{n+2}$ denotes de canonical inclusion. The second
fundamental forms and normal connections of $f$ and $\ft$ are
related as follows. Let $\Nb=\pi\circ i$ the unit normal vector
field to the inclusion $i$, where $\pi\co\E^{n+1}\times\R\to\E^{n+1}$
is the projection. We have
\[
\Nt_Z\Nb = \pi_\ast i_\ast Z = i_\ast(Z-\pl Z,\dt\pr\dt),
\]
for every $Z\in\X(\Qe)$, where $\Nt$ is the canonical connection
on $\E^{n+2}$. Hence,
\begin{equation} \label{eq:ShapeA^i}
A_{\Nb}^iZ = -Z + \pl Z,\dt\pr\dt.
\end{equation}
The normal spaces of $f$ and $\ft$ are related by
\[
TM^\perp_{\ft} = i_\ast TM^\perp_f\oplus\spa\{N\},
\]
where $N=\Nb\circ f=\pi\circ\ft$. Denoting by $\NB$ the Levi-Civita
connection of $\Qe$, and given $\xi\in TM^\perp_f$, it follows from
\eqref{eq:ShapeA^i} that
\[
\Nt_X i_\ast\xi =- \ft_\ast A^f_\xi X + i_\ast\nabla^\perp_X\xi +
\e\pl X,T\pr\pl\xi,\eta\pr N.
\]
In particular, we have
\begin{equation} \label{eq:ShapeA^i_2}
\Nt^\perp_X i_\ast\xi = i_\ast\nabla^\perp_X\xi + \e\pl X,T\pr\pl\xi,\eta\pr N,
\end{equation}
for every $\xi\in TM^\perp_f$, where $\Nt^\perp$ is the normal
connection of $\ft$. Moreover, 
\[
\Nt_XN = \ft_\ast(X-\pl X,T\pr T) - \pl X,T\pr i_\ast\eta
\]
which implies, in particular, that
\begin{equation} \label{eq:ShapeA^i_3}
\Nt^\perp_XN = -\pl X,T\pr i_\ast\eta.
\end{equation}

\section{Submanifolds with flat normal bundle in $\Qe$}
\label{sec:Flat}

An important class of submanifolds in $\Qe$ consists of those which 
have flat normal bundle. As in the case of submanifolds in a space
with constant sectional curvature, flatness of the normal bundle of a 
submanifold into $\Qe$ has the following useful characterization.

\begin{proposition} \label{prop:FlatNormal}
An isometric immersion $f\co M^m\to\Qe$ has flat normal bundle at
a point $x\in M^m$ if and only if the shape operators
\[
\{A_\xi: \xi\in T_xM^\perp\}
\]
are simultaneously diagonalizable. Equivalently, if and only if there exists
an orthonormal basis $\{X_1,\ldots,X_m\}$ of $T_xM$ such that
\[
\alpha_f(X_i,X_j) = 0, \ 1\leq i\neq j\leq m.
\]
\end{proposition}
\begin{proof}
By the Ricci equation \eqref{eq:Ricci2}, the normal curvature tensor
$R^\perp$ vanishes at $x\in M^m$ if and only if all shape operators
$A_\xi$, $\xi\in T_xM^\perp$, commute, and the conclusion follows.
\end{proof}

It follows from Proposition \ref{prop:FlatNormal} that, at each $x\in M$
where $R^\perp(x)=0$, the tangent space $T_xM$ decomposes 
orthogonally as
\begin{equation} \label{eq:DecompOrt}
T_xM=E_1(x)\oplus\cdots\oplus E_{s(x)}(x).
\end{equation}
This decomposition has the property that for each $\xi\in T_xM^\perp$,
there exist real numbers $\lambda_i(\xi)$, $1\leq i\leq s=s(x)$, such
that
\[
A_\xi\vert_{E_i(x)}=\lambda_i(\xi)\text{Id}
\]
and the maps $\xi\mapsto\lambda_i(\xi)$ are pairwise distinct. Since 
such maps are linear, there exist unique pairwise distinct vectors
$\xi_i(x)\in T_xM^\perp$, $1\leq i\leq s$, called the {\em principal
normals} of $f$ at $x$, such that
\[
\lambda_i(\xi)=\pl\xi_i(x),\xi\pr, \quad 1\leq i\leq s.
\]
Therefore, denoting by $E_i(x)=E_{\xi_i(x)}$, we have
\[
E_{\xi_i(x)}=\{X\in T_xM:\alpha_f(X,Y)=\pl X,Y\pr\xi_i(x), \ \forall \ Y\in T_xM\},
\]
and the second fundamental form of $f$ has the simple representation
\begin{equation} \label{eq:AlphaProj}
\alpha(X,Y)=\sum_{i=1}^{s}\pl X^i,Y^i\pr\xi_i(x),
\end{equation}
where $X\mapsto X^i$ is the orthogonal projection onto $E_i(x)$.
Equivalently,
\begin{equation} \label{eq:AProj}
A_\xi X = \sum_{i=1}^{s}\pl\xi,\xi_i(x)\pr X^i
\end{equation}
for all $X\in T_xM$ and $\xi\in T_xM^\perp$. 

\begin{remark}
Throughout this paper a submanifold with flat normal bundle will 
always be considered {\em proper}, that is, it has a constant number
$s=s(x)$ of principal normals vector fields. In this case, the principal
normal vector fields $x\in M^m\mapsto\xi_i(x)$, $1\leq i\leq s$, are
smooth. Moreover, the distributions $x\in M^m\mapsto E_i(x)$,
$1\leq i\leq s$, have constant dimension and are also smooth.
\end{remark}

\vspace{.2cm}

A direct calculation shows the following

\begin{lemma}
Let $f\co M^m\to\Qe$ be an isometric immersion with flat normal
bundle. Then, the Codazzi equation \eqref{eq:Codazzi} can be written 
as
\begin{equation} \label{eq:Codazziflat}
\langle X_j, Y_j \rangle (\nabla^{\perp}_{X_i}\xi_j+\epsilon\langle X_i,T\rangle\eta)=\langle \nabla_{X_j}Y_j,X_i \rangle (\xi_j-\xi_i), 
\end{equation}
for $1\leq i\neq j\leq s$, $ X_i\in \Gamma(E_i)$, and 
$X_j, Y_j \in \Gamma(E_j)$, or
\begin{equation} \label{eq:Codazziflat2}
\langle \nabla_{X_j}X_i,X_k\rangle(\xi_i-\xi_k)=\langle \nabla_{X_i}X_j,X_k
\rangle (\xi_j-\xi_k), 
\end{equation}
for $1\leq i\neq j \neq k \leq s$, $X_i\in \Gamma(E_i)$, 
$X_j\in \Gamma(E_j)$, and $X_k\in \Gamma(E_k)$.
\end{lemma}

Following the notations of \cite{MeTo}, we denote by $\A$ the class of 
isometric immersions $f\co M^m\to\Qe$ with the property that $T$ 
is an eigenvector of all shape operators of $f$. Our next result shows
a relationship between isometric immersions into $\Qe$ which have
flat normal bundle and those that are in class $\A$.

\begin{proposition} \label{prop:ClassA}
Let $f\co M^m\to\Qe$ be an isometric immersion with flat normal
bundle. Then, $f$ is in class $\A$ if and only if the tangent vector
field $T$ belongs to some $E_i$ in the decomposition 
\eqref{eq:DecompOrt}.
\end{proposition}
\begin{proof}
If $f$ belongs to class $\A$, then $A_\xi T=\lambda T$ for every
normal vector field $\xi\in\G(TM^\perp)$, where $\lambda=\lambda(\xi)$
is a smooth function along $M$. On the other hand, it follows from
\eqref{eq:AProj} that
\begin{equation} \label{eq:AProjT}
A_\xi T = \sum_{i=1}^{s}\pl\xi,\xi_i\pr T^i.
\end{equation}
Taking the inner product with $T^i\neq0$ of both sides of \eqref{eq:AProjT}
yields
\begin{equation} \label{eq:AProjT2}
\lambda\|T^i\|^2 = \pl\xi,\xi_i\pr\|T^i\|^2.
\end{equation}
Suppose that there exists another index $j\neq i$ such that $T^j\neq0$.
It follows from \eqref{eq:AProjT2} that $\pl\xi,\xi_i\pr=\pl\xi,\xi_j\pr$. In
particular, taking $\xi=\xi_i$ and then $\xi=\xi_j$, we obtain 
$\|\xi_i\|^2=\pl\xi_i,\xi_j\pr=\|\xi_j\|^2$. This implies that $\xi_i=\xi_j$,
which is a contradiction. Conversely, suppose that $T\in E_i$, for
some $1\leq i\leq s$. Thus
\[
\alpha(T,X) = \pl T,X\pr\xi_i,
\]
for every $X\in TM$. Therefore, given $\xi\in\G(TM^\perp)$, we have
$A_\xi T=\pl\xi,\xi_i\pr T$, and this shows that $T$ is an eigenvector
of $A_\xi$, for any $\xi\in\G(TM^\perp)$.
\end{proof}

\begin{remark}
Under the hypotheses of Proposition \ref{prop:ClassA}, we can
suppose, without loss of generality, that $T\in E_1$.
\end{remark}

\section{A basic result} \label{sec:Ricci}

Given an isometric immersion $f\co M^m\to\Qe$, we recall that the
{\em Ricci tensor} of $M^m$ is defined by
\[
\Ric(X,Y) = tr\{Z\mapsto R(Z,X)Y\}
\]
for every $X,Y\in TM$. In the next result we compute the Ricci
tensor of $M^m$ in terms of the second fundamental form of $f$.

\begin{lemma}
The Ricci tensor of an isometric imersion $f\co M^m\to\Qe$ is given
by
\begin{eqnarray} \label{eq:RicciTensor}
\begin{aligned}
\Ric(X,Y)  = & \ \e(m-1-\|T\|^2)\pl X,Y\pr + \e(2-m)\pl X,T\pr\pl Y,T\pr \\
& + m\pl H,\alpha(X,Y)\pr - III(X,Y),
\end{aligned}
\end{eqnarray}
where
\[
III(X,Y) = \sum_{i=1}^m\pl\alpha(X,X_i),\alpha(Y,X_i)\pr
\]
denotes the third fundamental form of $f$ in terms of an orthonormal 
tangent frame $\{X_1,\ldots,X_m\}$ of $M^m$.
\end{lemma}
\begin{proof}
Given an orthonormal tangent frame $\{X_1,\ldots,X_m\}$ of $M^m$,
the Gauss equation of $f$ yields
\begin{eqnarray*}
\Ric(X,Y)& = & \sum_{i=1}^{m}\langle R(X_i,X)Y,X_i\rangle \\
& = & \sum_{i=1}^{m}\big[\langle \epsilon(X_i\wedge X)Y
-\epsilon\langle X,T\rangle (X_i\wedge T)Y
+\epsilon \langle X_i,T\rangle (X\wedge T)Y \\
& & +\langle \alpha(X_i,X_i),\alpha(X,Y) \rangle
- \langle \alpha(X,X_i),\alpha(Y,X_i) \rangle\big] \\
& = & \epsilon(m-1)\langle X,Y \rangle-\epsilon\langle X,T \rangle
\big(m\langle Y,T \rangle-\langle Y,T \rangle\big)+\epsilon\langle Y ,T \rangle
\langle X,T\rangle \\
&& -\epsilon\langle X,Y \rangle\|T\|^2 
+ m\langle H,\alpha(X,Y) \rangle-III(X,Y) \\
& = & \epsilon\big(m-1-\|T\|^2\big)\langle X,Y \rangle+\epsilon(2-m)
\langle X,T \rangle\langle Y,T \rangle \\
&& +m\langle H, \alpha(X,Y)\rangle-III(X,Y),
\end{eqnarray*}
and the proof of \eqref{eq:RicciTensor} is completed.
\end{proof}

\begin{remark}
The formula for the Ricci tensor in \eqref{eq:RicciTensor} extends 
the one obtained in \cite{LePiSa} for the case of hypersurfaces.
\end{remark}

The {\em Ricci curvature} in the direction of a unit vector $X\in TM$
is defined as
\[
\Ric(X) = \frac{1}{m-1}\Ric(X,X).
\]
It follows from \eqref{eq:RicciTensor} that
\begin{eqnarray} \label{eq:RicciCurvTen}
\begin{aligned}
\Ric(X)  = & \ \e - \frac{\e}{m-1}\big(\|T\|^2-\pl T,X\pr^2\big) - \e\pl T,X\pr^2 \\
& + \frac{m}{m-1}\pl H,\alpha(X,X)\pr - \frac{1}{m-1}III(X,X),
\end{aligned}
\end{eqnarray}
for any unit vector $X\in TM$. Equation \eqref{eq:RicciCurvTen} yields
the following obstruction for the existence of a minimal submanifold
into any product space $\Qe$.

\begin{corollary}
Let $f\co M^m\to\Qe$ be a minimal isometric immersion. Then for any 
point $x\in M$, we have $\Ric(X)\leq\e$ for every unit vector $X\in T_xM$.
Moreover, the equality holds identically if and only if $f(M)$ is an open 
subset of a slice $N^m\times\{t_0\}$, where $N^m$ is a submanifold of
$\Qe$.
\end{corollary}
\begin{proof}
Since $f$ is minimal, we have
\[
\Ric(X)  =  \e - \frac{\e}{m-1}\big(\|T\|^2-\pl T,X\pr^2\big) - \e\pl T,X\pr^2 
- \frac{1}{m-1}III(X,X),
\]
and this proves the first statement, since $\|T\|^2-\pl X,T\pr^2\geq0$.
The equality holds identically if and only if $f$ is totally geodesic and
$T=0$, that is, $f(M)$ is an open subset of a slice of $\Qe$.
\end{proof}

\begin{corollary}
Let $f\co M^m_c\to\Qe$ be a minimal isometric immersion.
Then:
\begin{enumerate}
\item[(a)] $c\leq\e$,
\item[(b)] $c=\e$ if and only if $f(M)$ is an open subset of a slice 
in $\Qe$.
\end{enumerate}
\end{corollary}

The next result provides a characterization of Einstein submanifolds
with flat normal bundle, extending the one of \cite{LePiSa} for the
case of hypersurfaces.

\begin{theorem}
Let $f\co M^m\to\Qe$ be an Einstein submanifold, $m\geq3$, with flat 
normal bundle at a point $x\in M$. If $T\neq0$ at $x\in M$, then $T$ 
is an eigenvector of all shape operators of $f$ at $x$.
\end{theorem}
\begin{proof}
By Proposition \ref{prop:FlatNormal}, there exists an orthonormal basis
$\{X_1,\ldots,X_m\}$ of $T_xM$ such that
\[
\alpha_f(X_i,X_j) = 0, \ 1\leq i\neq j\leq m.
\]
Write $T=\sum_{k=1}^mt_kX_k$ at the point $x$. On the other hand, 
since $M$ is an Einstein manifold one has
\begin{equation} \label{eq:RicciA1}
\Ric(X_i,X_j) = \rho\delta_{ij},
\end{equation}
for some $\rho\in\R$. That is,
\[
\Ric(X_i,X_j) = \sum_{k=1}^m\pl R(X_k,X_i)X_j,X_k\pr = \rho\delta_{ij}.
\]
When applying the Ricci tensor to the referencial $\{X_1,\ldots,X_m\}$,
we obtain
\begin{eqnarray} \label{eq:RicciA2}
\begin{aligned}
\Ric(X_i,X_j)  = & \ \big(\e(m-1-\|T\|^2+m\lambda_i(H)\big)\delta_{ij} \\
& + \e(2-m)t_it_j - III(X_i,X_j).
\end{aligned}
\end{eqnarray}
It follows from \eqref{eq:RicciA1} and \eqref{eq:RicciA2} that
\begin{eqnarray} \label{eq:RicciA3}
\big(\e(m-1-\|T\|^2+m\lambda_i(H)-\rho\big)\delta_{ij} + \e(2-m)t_it_j 
- III(X_i,X_j) = 0.
\end{eqnarray}
Note that the third fundamental form of $f$, expressed in terms of
the basis $\{X_1,\ldots,X_m\}$, is always equal to zero, for $i\neq j$.
Thus, it follows from \eqref{eq:RicciA3} that
\[
t_it_j = 0, \ 1\leq i\neq j\leq m.
\]
Since $T(x)\neq0$, there is only one index $1\leq k\leq m$, with
$t_k\neq0$, and this implies that $T=t_kX_k$ at $x\in M$.
\end{proof}

\section{The proof of the main result}

Let $f\co M^m\to\Qe$, $m\geq3$, be an isometric immersion of a 
connected Einstein manifold with Ricci curvature $\lambda$, flat normal 
bundle, and parallel mean curvature vector field $H$. In order to prove
Theorem \ref{theo:main}, we will show a sequence of auxiliary results.

\begin{lemma} \label{lemanormaxik-h} 
If the tangent vector field $T$ is nowhere vanishing, then
\begin{equation} \label{eq:xi_normal}
\left\|\xi_k-\frac{m}{2}H\right\|^2 = \frac{m^2}{4}\left\|H\right\|^2 - \lambda
+ (m-2)\epsilon+\|\eta\|^2\epsilon,
\end{equation}
for any $2\leq k\leq s$, and for $k=1$ if $\dim E_1>1$.
\end{lemma}
\begin{proof}
Let $X_1=T/\|T\|$, and consider a unit vector field 
$X_2\in\G(E_1)\cap\{T\}^{\perp}$. Complete to an orthonormal frame
$\{X_1,X_2,\ldots,X_m\}$ of $M$, where each $X_i$ belongs to some
$\G(E_i)$. It follows from Gauss equation \eqref{eq:Gauss}, for $k=1$,
that
\begin{eqnarray*}
\lambda & = & \Ric(X_2,X_2) = \sum_{i=1}^{m}\pl R(X_i,X_2)X_2,X_i\pr \\
& = & \epsilon\|\eta\|^2+\|\xi_1\|^2(\dim(E_1)-1)+(m-2)\epsilon +
\sum_{i=2}^{s} \pl\xi_i,\xi_2\pr\dim(E_i),
\end{eqnarray*}
that is,
\[
\|\xi_1\|^2\dim(E_1) + \sum_{i=2}^{s}\pl\xi_i,\xi_1\pr\dim(E_i) = 
\lambda - (m-2)\epsilon + \|\xi_1\|^2 - \|\eta\|^2\epsilon.
\]
On the other hand, we can write 
\begin{eqnarray} \label{eq:mxi_1}
\begin{aligned}
m\pl\xi_1,H\pr & =  |\xi_1\|^2 \dim(E_1) + \sum_{i=2}^{s}\pl\xi_i,\xi_1\pr
\dim(E_i) \\
& =  \lambda-(m-2)\epsilon-\|\eta\|^2\epsilon +\|\xi_1\|^2.
\end{aligned}
\end{eqnarray}
Therefore, using \eqref{eq:mxi_1}, we have
\begin{eqnarray*}
\left\|\xi_1-\frac{m}{2}H\right\|^2 &=&
\left\pl\xi_1-\frac{m}{2}H,\xi_1-\frac{m}{2}H\right\pr \\
&=& \|\xi_1\|^2 -m\pl\xi_1,H\pr +\frac{m^2}{4}\|H\|^2 \\
&=& \frac{m^2}{4}\|H\|^2-\lambda+(m-2)\epsilon+\|\eta\|^2\epsilon.
\end{eqnarray*}
The case $k\geq2$ is analogous to the previous case. In fact, it
follows from Gauss equation \eqref{eq:Gauss} that
\begin{eqnarray*}
\pl\xi_1,\xi_k\pr\dim(E_1) + \sum_{i=2}^s\pl\xi_i,\xi_k\pr\dim(E_i) 
 = \lambda -(m-2)\e -\|\eta\|^2\e + \|\xi_k\|^2.
\end{eqnarray*}
The equation \eqref{eq:xi_normal} follows by writing 
\begin{equation} \label{eq:mxi_1_k}
m\pl\xi_k,H\pr = \lambda -(m-2)\e -\|\eta\|^2\e + \|\xi_k\|^2,
\end{equation}
as in the previous case.
\end{proof}

\begin{remark} \label{rem:constant}
It follows from Lemma \ref{lemanormaxik-h} that the functions
\[
h_k = \left\|\xi_k-\frac{m}{2}H\right\|^2
\]
are constant along $\{T\}^{\perp}$, for any $2\leq k\leq s$, and for $k=1$
if $\dim E_1>1$.
\end{remark}

\begin{lemma} \label{xikdupin}
The principal normals $\xi_1,\ldots,\xi_s$ of $f$ are parallel along 
$\{T\}^{\perp}$.
\end{lemma}
\begin{proof}
Given $X,Y\in\G(E_i)$ and $Z\in\G(E_k)\cap\{T\}^{\perp}$, with $2\leq i\leq s$
and $1\leq k\neq i\leq s$, it follows from Codazzi equation \eqref{eq:Codazziflat}
that
\begin{equation} \label{eq:Codazzixi}
\pl X,Y\pr\nabla^{\perp}_{Z}\xi_i = \pl\nabla_{X}Y,Z \pr(\xi_i-\xi_k).
\end{equation}
Since $H$ is parallel, we have
\[
\pl X,Y\pr\nabla^{\perp}_{Z}\left(\xi_i - \frac{m}{2}H\right) = 
\pl\nabla_{X}Y,Z \pr(\xi_i-\xi_k).
\]
This implies that
\begin{equation} \label{eq:Codazzixi2}
\pl X,Y\pr Z\left(\left\|\xi_i - \frac{m}{2}H\right\|^2\right) =
2\left\pl\nabla_{X}Y,Z \right\pr \left\pl\xi_i-\xi_k,\xi_i-\frac{m}{2}H\right\pr
\end{equation}
By Remark \ref{rem:constant}, the left-hand side of \eqref{eq:Codazzixi2}
is zero. Moreover, from \eqref{eq:xi_normal}, and using \eqref{eq:mxi_1}
and \eqref{eq:mxi_1_k}, one has
\[
\left\pl\xi_i-\xi_k,\xi_i-\frac{m}{2}H\right\pr = \frac{1}{2}\|\xi_i-\xi_k\|^2\neq0.
\]
This implies that $\pl\nabla_XY,Z\pr=0$, for all $X,Y\in\G(E_i)$, with
$2\leq i\leq s$. Thus, from \eqref{eq:Codazzixi}, one has
\[
\nabla^{\perp}_Z\xi_i=0,
\]
for all $Z\in\G(E_k)\cap\{T\}^{\perp}$, with $2\leq k\neq i\leq s$.
For $Z\in\G(E_i)$, we obtain
\begin{eqnarray*}
0 &=& \nabla^{\perp}_ZH=\sum_{i\neq k=1}^{s}\dim(E_k)
\nabla^{\perp}_Z\xi_k+\dim(E_i)\nabla^{\perp}_Z\xi_i \\
&=& \dim(E_i)\nabla^{\perp}_Z\xi_i.
\end{eqnarray*}
This shows that $\xi_i$ is parallel along $\{T\}^{\perp}$, for $2\leq i\leq s$.
For the case $i=1$, it follows from Codazzi equation \eqref{eq:Codazziflat}
that
\[
\|T\|^2\nabla^{\perp}_Z\xi_1=\langle \nabla_TT,Z\rangle (\xi_1-\xi_k) = 0,
\]
for every $Z\in \Gamma(E_k)$, with $2\leq k\leq s$, that is, 
$\nabla^{\perp}_Z\xi_1=0$, for every $Z \in \Gamma(E_k)$. 
Finally, if there exists $X \in \Gamma(E_1)\cap \{T\}^{\perp}$, then 
\[
\dim(E_1)\nabla^{\perp}_X\xi_1 = \nabla^{\perp}_XH 
- \sum_{k=2}^{s}\dim(E_k)\nabla^{\perp}_X\xi_k = 0,
\]	
by the previous case. This shows that $\xi_i$ is parallel along $\{T\}^{\perp}$.
\end{proof}

Given an isometric immersion $f\co M^m\to\Qe$ as it has been fixed,
consider the composition $\ft=i\circ f$. In this case, it follows from 
\cite[Corollary 1.3]{MeTo} that $\ft$ also has flat normal bundle.
In order to relate the flatness of the normal bundle of $f$ and $\ft$,
let $\Nb=\pi\circ i$ the unit normal vector field to the inclusion $i$,
and consider the decomposition \eqref{eq:DecompOrt} for
$f$ with principal normal vector fields $\xi_1,\ldots,\xi_s$, where
\[
E_i(x)=\{X\in T_xM:\alpha_f(X,Y)=\pl X,Y\pr\xi_i(x), \ \forall \ Y\in T_xM\}.
\]

\begin{lemma}
The following assertions hold:
\begin{enumerate}
\item[(a)] If $\dim E_1=1$, then \eqref{eq:DecompOrt} is also a 
decomposition for $\tilde{f}$, with principal normal vector fields
given by
\[
i_{\ast}\xi_1-\|\eta\|^2\Nb, i_{\ast}\xi_2-\bar{N},\dots,i_{\ast}\xi_s-\bar{N}.
\]
\item[(b)] 
If $\dim E_1\geq 2$, then
\[
T_xM =\spa\{T(x)\}\oplus\big(E_1(x)\cap\spa\{T(x)\}^\perp\big)\oplus
E_2(x)\oplus\ldots\oplus E_s(x),
\]
is a decomposition for $\tilde{f}$, whose principal normals are given
by
\[
i_{\ast}\xi_1-\|\eta\|^2\bar{N}, i_{\ast}\xi_1-\bar{N}, i_{\ast}\xi_2-\bar{N},\dots,i_{\ast}\xi_s-\bar{N}.
\]
\end{enumerate}
\end{lemma}
\begin{proof}
Take any vector $X\in E_k(x)$, with $1\leq k\leq s$ and $X\in\{T(x)\}^\perp$.
Then
\begin{eqnarray*}
\alpha_{\ft}(X,Y) &=& i_\ast\alpha_f(X,Y) + \alpha_i(f_\ast X,f_\ast Y) \\
&=& i_\ast\pl X,Y\pr\xi_k - \pl X,Y\pr\Nb + \pl X,T\pr\pl Y,T\pr\Nb \\
&=& \pl X,Y\pr(i_\ast\xi_k-\Nb),
\end{eqnarray*}
for all $Y\in TM$. Moreover, 
\begin{eqnarray*}
\alpha_{\ft}(T,Y) &=& i_\ast\alpha_f(T,Y) + \alpha_i(f_\ast T,f_\ast Y) \\
&=& i_\ast\pl T,Y\pr\xi_1 - \pl T,Y\pr\Nb + \pl T,T\pr\pl T,Y\pr\Nb \\
&=& \pl T,Y\pr(i_\ast\xi_1-\|\eta\|^2\Nb),
\end{eqnarray*}
for all $Y\in TM$, and this concludes the proof.
\end{proof}

From now on we write $E_0=\spa\{T\}$ and, with abuse of notation, 
we write $E_1=E_1\cap E_0^\perp$. Thus, at each point $x\in M$,
we have the decomposition for $\ft$ given by
\begin{equation} \label{eq:decomp_ft}
T_xM = E_0(x)\oplus E_1(x)\oplus\ldots\oplus E_s(x),
\end{equation}
whose principal normals are given by
\[
i_{\ast}\xi_1-\|\eta\|^2\bar{N}, i_{\ast}\xi_1-\bar{N},\dots,
i_{\ast}\xi_s-\bar{N},
\]
respectively, $E_1=\{0\}$ may occur.

\begin{lemma} \label{lem:SphTotGeod}
The distributions $\{E_k\}$ of the decomposition \eqref{eq:decomp_ft}
are such that $E_k$ is spherical and $E_k^\perp$ is totally geodesic,
for all $1\leq k\leq s$. 
\end{lemma}
\begin{proof}
Given $X,Y\in\G(E_k)$, $1\leq k\leq s$, we have
\begin{eqnarray*}
\pl\nabla_XY,T\pr &=& -\pl Y,\nabla_XT\pr = - \pl T,A_\eta X\pr =
-\pl Y,A_{i_\ast X}^{\ft}X\pr \\
&=& - \pl\alpha_{\ft}(X,Y),i_\ast\eta\pr 
=- \pl X,Y\pr\pl i_\ast\xi_k-\Nb,i_\ast\eta\pr \\
&=& -\pl X,Y\pr\pl\xi_k,\eta\pr =- \pl X,Y\pr\pl T,\pl\xi_k,\eta\pr T\pr
\frac{1}{\|T\|^2}
\end{eqnarray*}
that is,
\begin{equation} \label{Ekumbilica}
\pl\nabla_XY, T\pr = \pl X,Y\pr\pl T,\delta\pr,
\end{equation}
where
\begin{equation} \label{Ekumbilica3}
\delta=-\frac{1}{\|T\|^2}\langle\xi_k,\eta\rangle T\in\G(E_k^{\perp}).
\end{equation}
Moreover, for $Z\in\G(E_i)$, with $1\leq i\neq k\leq s$,
Lemma \ref{xikdupin} yields $\nabla^{\perp}_Z\xi_k=0$. Thus, it follows
from Codazzi equation \eqref{eq:Codazziflat} that
\begin{equation} \label{Ekumbilica2}
\langle \nabla_XY,Z\rangle=0=\langle X,Y\rangle\langle Z, \delta\rangle.
\end{equation}
Equations \eqref{Ekumbilica} and \eqref{Ekumbilica2} imply that
$E_k$ is an umbilical distribution with mean curvature vector field
$\delta$ given in \eqref{Ekumbilica3}. On the other hand, it follows
from Lemma \ref{xikdupin} and the fact that $\|T\|$ is contant along
$\{T\}^\perp$, that
\[
\pl\nabla_X\delta,Z\pr = -\frac{1}{\|T\|^2}\pl T,Z\pr\pl\eta,
\nabla^\perp_X\xi_k\pr=0,
\]
for all $X\in \Gamma(E_k)$ and $Z\in \Gamma(E_k^{\perp})$, and this
shows that $E_k$ is a spherical distribution. In order to show that 
$E_k^\perp$ is totally geodesic, it suffices to prove that
\begin{equation} \label{eq:E_k^perp}
\pl\nabla_XY,Z\pr = 0,
\end{equation}
for all $X\in\G(E_i)$, $Y\in\G(E_j)$ and $Z\in\G(E_k)$, with 
$1\leq k\leq s$, $0\leq i,j\leq s$, $k\neq\{i,j\}$. For $0\leq i=j\leq s$,
equation \eqref{eq:E_k^perp} follows from Lemma \ref{xikdupin}
and Codazzi equation. Now, let $0\leq i\neq j\leq s$, and let us
consider the following cases.
\vspace{.2cm} \\
Case 1: If $j=0$ and $Y=\lambda T$, for some smooth function $\lambda$
along $M$, then 
\begin{eqnarray*}
\pl\nabla_XY,Z\pr &=& \pl\nabla_X\lambda T,Z\pr = 
\pl X(\lambda)T+\lambda\nabla_XT,Z\pr \\
& = & \lambda\langle A_{\eta}X,Z\rangle=0.
\end{eqnarray*}
Case 2: If $i=0$ and $X=\lambda T$, for some smooth function $\lambda$
along $M$, it follows from Codazzi equation \eqref{eq:Codazziflat2} that
\begin{eqnarray*}
\pl\nabla_XY,Z\pr(\xi_j-\xi_k) &=& \pl\nabla_{\lambda T}Y,Z\pr(\xi_j-\xi_k) \\
& = & \langle \nabla_{Y}\lambda T,Z\rangle (\xi_i-\xi_k) \\
&=& 0
\end{eqnarray*}
as in the previous case.
\vspace{.2cm} \\
Case 3: Suppose now $1\leq i\neq j\leq s$ and let $X,Y\in\{T\}^{\perp}$.
It follows from Codazzi equation \eqref{eq:Codazziflat2} that
\begin{equation} \label{eqflat2li}
\langle \nabla_{X}Y,Z\rangle(\xi_j-\xi_k)=
\langle \nabla_{Y}X,Z\rangle (\xi_i-\xi_k).
\end{equation}
We claim that the vector fields $\xi_i-\xi_k$ and $\xi_j-\xi_k$ are linearly
independent. Indeed, assume otherwise that there exists a smooth
nonzero function $u$ along $M$ such that
\[
\xi_i-\xi_k=u(\xi_j-\xi_k).
\]
It follows that
\[
(u-1)\left(\xi_k-\frac{m}{2}H\right) =
u\left(\xi_j-\frac{m}{2}H\right) - \left(\xi_i-\frac{m}{2}H\right).
\]
By taking the norms, it follows from \eqref{eq:xi_normal} that
\begin{equation} \label{Ekumbilica4}
\left\|\xi_j-\frac{m}{2}H\right\|^2 = 
\left\pl\xi_j-\frac{m}{2}H,\xi_i-\frac{m}{2}H \right\pr.
\end{equation}
Now, using \eqref{eq:mxi_1} and \eqref{eq:mxi_1_k} in \eqref{Ekumbilica4},
we obtain $\|\xi_i=\xi_j\|$, that is, $\xi_i=\xi_j$, which is a contradiction.
The conclusion now follows from \eqref{eqflat2li}.
\end{proof}

Finally, we can prove the main result of this paper.

\begin{proof}[Proof of Theorem \ref{theo:main}]
Let $f\co M^m\to\Se$, $m\geq3$, be a proper isometric immersion
of a connected Einstein manifold, with flat normal bundle and parallel
mean curvature vector field $H$. Moreover, suppose that the vector
field $T$ given in \eqref{eq:dt} is nowhere vanishing. Let $\ft=i\circ f$ 
be considered as a submanifold
in the Euclidean space $\R^{n+2}$, and consider the distributions 
$(E_k)$, $0\leq k\leq s$, of $M^m$
given in \eqref{eq:decomp_ft} which, by virtue of Lemma 
\ref{lem:SphTotGeod}, is such that $E_k$ is spherical and 
$E_k^\perp$ is totally geodesic, for $1\leq k\leq s$. It follows
from theorem of the Hiepko \cite{Hi} that there exists locally (globally,
if $M^m$ is simply connected and complete) a product 
representation
\[
\psi\co M_0\times M_1\times\ldots\times M_s\to M^m
\]
which is an isometry with respect to a warped product metric on the
product $M_0\times M_1\times\ldots\times M_s$, where $M_0$ is
totally geodesic and $M_k$ are intrinsic spheres, for $1\leq k\leq s$.
In this case, as $\dim(M_0)=1$, we can suppose that $M_0$ is an
open interval $I_0\subset\R$. Consider now the immersion
\[
F=\ft\circ\psi\co I_0\times_{\rho_1}M_1^{n_1}\times \ldots\times_{\rho_s}
M_s^{n_s}\to\R^{n+2}.
\]
$F$ is an isometric immersion defined in a warped product manifold,
whose second fundamental form is adapted to the product net
\eqref{eq:decomp_ft} of $M^m$. It follows from Nolker's theorem
\cite{No} that $F$ is an extrinsic warped product of isometric
immersions. More precisely, there exist a warped product
representation 
\[
\phi\co\R^{s+1}\times_{\sigma_1}N_1^{n_1+k_1}\times\ldots
\times_{\sigma_s}N_s^{n_s+k_s}\to\R^{n+2},
\]
isometric immersions $f_j\co M_j^{n_j}\to N_j^{n_j+k_j}$, $1\leq j\leq s$, 
and an smooth curve $f_0\co I_0\to\R^{s+1}$ such that
$\rho_j=\sigma_j\circ f_0$ for $1\leq j\leq s$ and
\[
F = \phi\circ(f_0\times\ldots\times f_s),
\]
where $N_j^{n_j+k_j}$ are Euclidean spheres. Moreover, each
isometric immersion $f_j\co M_j^{n_j}\to N_j^{n_j+k_j}$ is umbilical,
with mean curvature vector fiel given by $H_j=i_\ast\xi_j-\Nb$. 
Each of these normal vector fields is parallel along the normal
connection $\Ntj^\perp$ of $f_j$, which is the restriction of 
$\Nt^\perp$ along $M_j^{n_j}$. In fact, using Lemma \ref{xikdupin},
and equations \eqref{eq:ShapeA^i_2} and \eqref{eq:ShapeA^i_3},
we obtain
\begin{eqnarray*}
\Ntj^\perp_XH_j &=& \Ntj^\perp_Xi_\ast\xi_j - \Ntj^\perp_X\Nb \\
&=& i_\ast\nabla^\perp_X\xi_j + \pl X,T\pr\pl\xi_j,\eta\pr\Nb +
\pl X,T\pr i_\ast\eta \\
&=& 0,
\end{eqnarray*}
for every $X\in\X(M_j^{n_j})$. Therefore, each $f_j$ reduce
codimension to one, and by rigidity, all of them must be the
canonical inclusion. We conclude that $f$ is a multi-rotational
submanifold with profile $f_0$, where $f_0$ is a smooth
curve determined by the vector field $T$.
\end{proof}

\bibliographystyle{amsplain}

\Addresses

\end{document}